\documentclass[11pt]{amsart}
\usepackage{mathrsfs}
\usepackage{amssymb}
\usepackage{caption}
\usepackage{nicematrix}
\usepackage{enumerate}
\usepackage{booktabs}
\usepackage{geometry} 
\usepackage{amsbsy}
\newtheorem{theorem}{Theorem}[section]
\newtheorem{lemma}[theorem]{Lemma}

\theoremstyle{definition}
\newtheorem{definition}[theorem]{Definition}
\newtheorem{example}[theorem]{Example}

\theoremstyle{remark}

\numberwithin{equation}{section}

\usepackage{cite}

\begin{document}
	
	\begin{center}
		{\Large \bf Open-loop Pareto-Nash equilibria in multi-objective interval differential games}
	\end{center}

	\vskip 20pt

	\begin{center}
		{{\bf Wen~~Li\footnote{Corresponding author, E-mail:liwenmath@wust.edu.cn
				},
				\quad Du Zou \footnote{zoudumath@wust.edu.cn},
			    \quad Deyi Li \footnote{lideyi@wust.edu.cn},
				\quad Yuqiang Feng 
				\footnote{yqfeng6@126.com}}\\~~
			
			School of Science, Wuhan University of Science and Technology, Wuhan, 430065, China\\
		    }
	\end{center}
	
	\vskip 10pt
	\footnotetext{Research is supported by the Excellent Youth Foundation of Hubei Scientiﬁc Committee (No. 2020CFA079) and the National Natural Science Foundation of China (Nos. 72031009, 12171378).}
	
	\begin{center}
		\begin{minipage}{16cm}
			{\bf Abstract} 
The paper explores $n$-player multi-objective interval differential games, where the terminal payoff function and integral payoff function of players are both interval-vector-valued functions. Firstly, by leveraging the partial order relationship among interval vectors, we establish the concept of (weighted) open-loop Pareto-Nash equilibrium for multi-objective interval differential games and derive two theorems regarding the existence of such equilibria. Secondly, necessary conditions for open-loop Pareto-Nash equilibria in $n$-player interval differential games are derived through constructing Hamilton functions in an interval form and applying the Pontryagin maximum principle. Subsequently, sufficient conditions for their existence are provided by defining a maximization Hamilton function and utilizing its concavity. Finally, a two-player linear quadratic interval differential game is discussed along with a specific calculation method to determine its open-loop Pareto-Nash equilibrium.
			
			\vskip 10pt
		\end{minipage}
	\end{center}

\vskip 25pt
\section{\bf Introduction}
\vskip 10pt
Differential games originated in the 1950s as a result of studying the challenges faced by opposing sides in air force confrontations. In 1965, Isaacs\cite{Isaacs} authored the renowned book "Differential Games," which marked a significant milestone in the development of this field. Differential games primarily involve game theories that utilize state equations to describe dynamic systems. These theories include the saddle point equilibrium theory for zero-sum differential games and the Nash equilibrium theory for non-zero-sum differential games. Since the seminal contributions of Isaacs and Starr and Ho\cite{StarrHo}, differential games have played a crucial role in various domains such as economics, military defense, and management\cite{Isaacs,Dockner}. Differential games are widely observed in the domain of control engineering, covering areas such as robust control\cite{James,Limebeer,Mylvaganam}, robotics\cite{Cappello,Mylvaganam,XueZhanWu}, power systems\cite{SaadHan,MeiZhang}, and collaborative control challenges\cite{Semsar-Kazerooni,GaoPetrosyan}.

The robust control problem can be transformed into a saddle point equilibrium or Nash equilibrium problem of differential games, as extensively discussed in the works of Basar and Bernhard\cite{BasarBernhard}, Broek, Engwerda, and Schumacher\cite{Broek}. Therefore, the theory of differential games can be seen as an extension of optimal control. Linear quadratic differential games have attracted significant attention due to their solvability through analytical and numerical methods. Xu and Mizukami\cite{XuMizukami} examined zero-sum linear quadratic differential games for generalized state space systems and established sufficient conditions for the existence of a linear feedback saddle point equilibrium. Basar and Olsder\cite{BasarOlsder} systematically summarized differential games with quadratic performance indicator functions in linear systems, providing necessary conditions for describing game theories such as saddle point equilibrium and Nash equilibrium in general nonlinear systems using the Hamilton-Jacobi-Bellman equation or the Hamilton-Jacobi-Isaacs equation. Huang and Zhao\cite{HuangZhao} investigated switching linear quadratic differential games within a finite range by introducing preference information and auxiliary variables to equivalently transform them into a series of parameterized single-objective optimal problems. They then utilized dynamic programming techniques to construct the Pareto frontier of these switching linear quadratic differential games. Zhukovskiy et al.\cite{Zhukovskiy} developed coefficient criteria based on dynamic programming and small parameter methods to determine the existence of Berge equilibrium and Nash equilibrium in linear quadratic differential games.

Due to the presence of uncertainties in various real-world phenomena and systems, individuals often employ random variables and processes to describe unknown internal and external disturbances and errors. This necessitates the utilization of stochastic differential equations as a dynamic model for these systems, leading to the emergence of stochastic differential games. For further exploration of stochastic differential games, please refer to references\cite{Basar,Hamadene1998,Hamadene1999,HamadeneMu,SunYong}. Similarly, fuzziness represents another form of uncertainty that holds significant importance in reality. In 1965, Zadeh\cite{Zadeh} initially introduced fuzzy sets based on membership functions. To quantify a fuzzy event, Liu and Liu\cite{LiuLiu} introduced the concept of credibility measure in 2002. Building upon his uncertainty theory established in 2004, Liu\cite{Liu2007} further enhanced it. Additionally, Liu\cite{Liu2008} proposed the notion of fuzzy processes. Subsequently, some scholars delved into studying fuzzy differential equations\cite{ChenQin,YouWang}, as well as fuzzy optimal control problems\cite{Zhu,ZhaoZhu,QinBai,BatenKamil}. Drawing from Liu's earlier concept of the fuzzy process presented above\cite{Liu2008}, Hu and Yang investigated the existence of open-loop Pareto-Nash equilibrium for differential games described by fuzzy differential equations.

In practical scenarios where precise numerical values are difficult to obtain due to factors like measurement errors, incomplete information, data fluctuations, and uncertainties; interval numbers serve as a valuable approach for describing these uncertainties. Moore\cite{Moore} introduced the concept of interval numbers and employed interval analysis to analyze functions with interval coefficients. Subsequent studies by researchers focused on exploring the theory and applications of interval analysis (see references\cite{Alefeld,Chalco-Cano,Oppenheimer,Hukuhara,Stefanini}). Oppenheimer and Michel\cite{Oppenheimer} demonstrated that systems utilizing Moore's interval operations cannot form a group since there is no inverse element for nonzero intervals. To overcome this limitation, Hukuhara\cite{Hukuhara} proposed the concept of $H$-difference as an alternative measure for quantifying differences between intervals. However, the $H$ difference does not define the difference between narrow-width and wide-width intervals. Addressing this concern, Stefanni and Bede\cite{Stefanini} introduced the notion of $gH$-difference as an extension of $H$-difference. Building upon this definition, they also introduced the concept of generalized Hukuhara derivative ($gH$-derivative) of the interval-valued function. Currently, $gH$-difference and $gH$-derivative are widely employed in solving problems related to interval optimization (references\cite{Antczak,Ghosh2020,Ghosh2022,Li2023}), interval optimal control\cite{Leal} [46], as well as interval differential equations (references\cite{Stefanini,Lupulescu,TaoZhang}).

Leal et al.\cite{Leal} investigated the optimal control problem involving interval-valued objective functions by utilizing extreme values and generalized Hukuhara differentiability of interval-valued functions to derive necessary and sufficient conditions for optimality. Li et al.\cite{LiLi} explored weak Pareto-Nash equilibrium in generalized interval-valued static multi-objective games with fuzzy mappings, establishing an existence theorem for weak Pareto-Nash equilibrium and examining its applications in multi-objective bi-matrix games. Building upon this prior research, our focus is on dynamic multi-objective differential games. To better capture the uncertainty associated with payoff values in these games, we employ interval-vector-valued functions to represent the payoff function. Consequently, this paper concentrates on a differential game where both the terminal payoff function and integral payoff function are represented as interval-vector-valued functions (referred to as a multi-objective interval differential game), aiming to identify its open-loop Pareto-Nash equilibrium.

The paper is organized as follows: In Section 2, we introduce several fundamental concepts and relevant theorems. Section 3 presents the existence theorems for open-loop Pareto-Nash equilibrium and weighted open-loop Pareto-Nash equilibrium in multi-objective interval differential games, respectively. Moving on to Section 4, our focus shifts to the $n$-person interval differential game and its open-loop Pareto-Nash equilibrium. By employing interval-valued Hamiltonian functions, Pontryagin maximum principle, and exploiting concavity of the maximization Hamiltonian function, we establish sufficient and necessary conditions for the existence of open-loop Pareto-Nash equilibrium. In Section 5, we provide explicit calculations for obtaining the open-loop Pareto-Nash equilibrium of linear quadratic interval differential games. Finally, in Section 6, we conclude this article.

\vskip 25pt
\section{\bf Preliminaries and Terminologies}
\vskip 10pt

This section primarily presents various valuable concepts and significant theorems, such as the partial ordering of interval vectors, $gH$-derivatives for interval-valued functions, integration of interval-vector-valued functions, and more.

Let $\mathcal{I}(\mathbb{R})^d$ be the set of all $d$-dimensional interval vectors. Let $A_i = [\underline{a}_i, \overline{a}_i]$, $B_i = [\underline{b}_i, \overline{b}_i]$, $i = 1, \dots, d$, for $\boldsymbol{A} = (A_1, \ldots, A_d)$, $\boldsymbol{B} = (B_1, \ldots, B_d) $ and $\lambda \in \mathbb{R}$, then $\boldsymbol{A} + \boldsymbol{B} = (A_1 + B_1, \ldots, A_d + B_d)$ and $\lambda \cdot \boldsymbol{A} = (\lambda \cdot A_1, \ldots, \lambda \cdot A_d)$. The $gH$-difference between $\boldsymbol{A}$ and $\boldsymbol{B}$, denoted $\boldsymbol{A} \ominus_{gH} \boldsymbol{B}$, is defined as $\boldsymbol{A} \ominus_{gH} \boldsymbol{B} = (A_1 \ominus_{gH} B_1, \ldots, A_d \ominus_{gH} B_d)$, where $A_i \ominus_{gH} B_i = [\min\{\underline{a}_i - \underline{b}_i, \overline{a}_i - \overline{b}_i\}, \max\{\underline{a}_i - \underline{b}_i, \overline{a}_i - \overline{b}_i\}]$.

Let $\mathcal{I}(\mathbb{R}_+)$ be the set of all compact intervals contained in $[0, +\infty)$, and 
\[
\mathcal{I}(\mathbb{R}_+)^d = \left\{ (A_1, \ldots, A_d)\ |\ A_i \in \mathcal{I}(\mathbb{R}_+), i=1, \ldots, d \right\}.
\] 
 
\begin{definition}[See \cite{LiLi}]\label{definition2.1}
	For $\boldsymbol{A}, \boldsymbol{B} \in \mathcal{I}(\mathbb{R})^d$, $\boldsymbol{B}$ is said to strictly dominate $\boldsymbol{A}$ from below if $\boldsymbol{A} \ominus_{gH} \boldsymbol{B} \in int \mathcal{I}(\mathbb{R_+})^d$, denoted as $\boldsymbol{B} \prec \boldsymbol{A}$; otherwise, it is denoted as $\boldsymbol{B} \nprec \boldsymbol{A}$.
\end{definition}

\begin{definition}[See \cite{LiLi}]\label{definition2.2}
	Let $\mathcal{X}$ be a Hausdorff topological space, and let $F : \mathcal{X} \to \mathcal{I}(\mathbb{R})$ be an interval-valued function. Then $F$ is said to be continuous if, for each $x \in \mathcal{X}$ and for each $\epsilon > 0$, there exists an open neighborhood $o(x)$ of $x$ such that for each $x' \in o(x)$, $F(x') \ominus_{gH} F(x) \subset (-\epsilon, \epsilon)$. 
\end{definition}

\begin{definition}[See \cite{LiLi}]\label{definition2.3}
	Let $\mathcal{X}$ be a Hausdorff topological space, and $\mathcal{K}$ be a nonempty convex subset of $\mathcal{X}$. The interval-valued function $F: \mathcal{K} \to \mathcal{I}(\mathbb{R})$ is said to be \textit{generalized $\mathcal{I}(\mathbb{R_+})$-quasi-concave} on $\mathcal{K}$ if, for $x_1, x_2 \in \mathcal{K}$, $\boldsymbol{A} \in \mathcal{I}(\mathbb{R})$, and for $\lambda \in [0,1]$, $\boldsymbol{A} \ominus_{gH} F(x_i) \notin int \mathcal{I}(\mathbb{R_+})$ with $i = 1, 2$,
	we have $\boldsymbol{A} \ominus_{gH} F(\lambda x_1 + (1-\lambda) x_2) \notin int \mathcal{I}(\mathbb{R_+})$.
\end{definition}

\begin{lemma}[See \cite{LiLi}]\label{lemma2.4}
	Let $\mathcal{X}_1$ and $\mathcal{X}_2$ be two Hausdorff topological spaces, $X_i$ be a nonempty compact subset of $\mathcal{X}_i\ (i = 1,2)$. And let $F : X_1 \times X_2 \to \mathcal{I}(\mathbb{R})$ be a continuous interval-valued function. Define the set-valued mapping $\Phi: X_2 \rightrightarrows X_1$ as follows: for any $y \in X_2$,
	\[
	\Phi(y) = \{x \in X_1\ |\ F(x,y)\nprec F(u, y),\ \forall\ u \in X_1 \}
	\]
	Then $\Phi$ is nonempty compact-valued and upper semicontinuous.
\end{lemma}

\begin{theorem}[Fan-Glicksberg\cite{Fan,Glicksberg}]\label{theorem2.5}
	Let $X$ be a nonempty compact subset of a Hausdorff locally convex topological vector space $\mathcal{Y}$. If $\Phi : X \rightrightarrows X$ is an upper-semicontinuous set-valued mapping with nonempty convex and compact values, then there exists $x^* \in X$ such that $x^* \in \Phi(x^*)$.
\end{theorem}

\begin{definition}[See\cite{Stefanini}]\label{definition2.6}
	Let $F : (t_0, t_1) \to \mathcal{I}(\mathbb{R})$ be an interval-valued function. For $t \in (t_0, t_1)$ and $t + h \in (t_0, t_1)$, if $\lim_{h \to 0} \frac{F(t + h) \ominus_{gH} F(t)}{h} $ exists, then $F$ is said to be $gH$-differentiable at $t$. The $gH$-derivative of $F$ is denoted by
	\[
	F'(t) = \lim_{h \to 0} \frac{F(t + h) \ominus_{gH} F(t)}{h}.
	\]
\end{definition}

The interval-valued function $F$ is said to be $gH$-right-differentiable at $t$ if the limit of $\frac{F(t + h) \ominus_{gH} F(t)}{h}$ as $h \to 0^+$ exists, and it is said to be $gH$-left-differentiable at $t$ if the limit of $\frac{F(t + h) \ominus_{gH} F(t)}{h}$ as $h \to 0^-$ exists. Furthermore, $F$ is said to be $gH$-differentiable in the interval $[t_0, t_1]$ if it is differentiable at every point in $(t_0, t_1)$ and both $gH$-right-differentiable at $t_0$ and $gH$-left-differentiable at $t_1$.

\begin{theorem}[See\cite{Chalco-Cano}]\label{theorem2.7}
	Let $F :[t_0, t_1] \to \mathcal{I}(\mathbb{R})$ be and interval-valued function. Assuming that $F$ is $gH$-differentiable at $t \in [t_0, t_1]$, then $F$ is continuous at $t \in [t_0, t_1]$.
\end{theorem} 

\begin{theorem}[See\cite{Chalco-Cano}]\label{theorem2.8}
	Let $F :[t_0, t_1] \to \mathcal{I}(\mathbb{R})$ be an interval-valued function, where $F(t) = [\underline{f}(t), \overline{f}(t)]$. Then, $F$ is $gH$-differentiable if and only if one of the following conditions holds:
	\begin{enumerate}
		\item[(1)] $\underline{f}$ and $\overline{f}$ is differentiable at $t \in [t_0, t_1]$,
		\item[(2)] the one-sided derivatives $\underline{f}'_-(t)$, $\underline{f}'_+(t)$, $\overline{f}'_-(t)$ and $\overline{f}'_+(t)$ all exist, and satisfy $\underline{f}'_-(t) = \overline{f}'_+(t)$, $\underline{f}'_+(t) = \overline{f}'_-(t)$.
	\end{enumerate}
\end{theorem}

The definition of the $gH$-partial derivative of an $n$-variable interval-valued function is given below, with the assistance of the $gH$-derivative of the interval-valued function.

Let $F : X \subset \mathbb{R}^n \to \mathcal{I}(\mathbb{R})$ be an $n$-variable interval-valued function, and let $x_0 = (x_1^0, \dots, x_n^0) \in X$. For each $i$, the interval-valued function $h_i : X_i \subset \mathbb{R} \to \mathcal{I}(\mathbb{R})$ is defined as follows.
\[
h_i(x_i) = F\big(x_1^0, \dots, x_{i-1}^0, x_i, x_{i+1}^0, \dots, x_n^0\big).
\]
If $h_i$ is $gH$-differentiable at $x_i^0$, then $F$ has the $i$th $gH$-partial derivative at $x_0$, denoted by
\[
\Big(\frac{\partial F}{\partial x_i} \Big)_{gH} (x_0) = h'_i(x_i^0).
\]
Moreover, if $F$ is continuous at $x_0$ and all $gH$-partial derivatives exist in the neighborhood of $x_0$, then $F$ is both continuous and $gH$-differentiable at $x_0$.

\begin{definition}[See \cite{Leal}]\label{definition2.9}
	Let $F$ be an $n$-variable interval-valued function on $X \subset \mathbb{R}^n$, where $F(x) = [\underline{f}(x), \overline{f}(x)]$. Then $F$ is end-point differentiable ($E$-differentiable) at $x_0$ if and only if the real-valued functions $\underline{f}$ and $\overline{f}$ are differentiable at $x_0$. 
\end{definition}

\begin{theorem}[See \cite{Leal}]\label{theorem2.10}
	Let $F$ be an $n$-variable interval-valued function on $X \subset \mathbb{R}^n$, where $F(x) = [\underline{f}(x), \overline{f}(x)]$. Assuming that $F$ is continuously $E$-differentiable, then
	\begin{enumerate}
		\item[(1)] $F$ is continuously $gH$-differentiable;
		\item[(2)] $\underline{f}'(x), \overline{f}'(x) \in F'(x)$.
	\end{enumerate}
\end{theorem}

Let $F = \{f_1, \dots, f_d\}:[t_0, t_1] \to \mathcal{I}(\mathbb{R})^d$ be an interval-vector-valued function, and let $\mathcal{S}(F)$ be the set of integrable selectable vector-valued functions of $F$ on $[t_0, t_1]$, that is,
\[
\mathcal{S}(F) = \big\{F^*: [t_0, t_1] \to \mathbb{R}^d\ |\ F^*\ \mbox{is integrable}\ \mbox{and}\  F^*(t) \in F(t)\ \mbox{a.e.}\big\}
\]

The integral of $F$ on the interval $[t_0, t_1]$ is defined as follows:
\begin{equation*}
		\int_{t_0}^{t_1} F(x) \mathrm{d}t = \Big\{\int_{t_0}^{t_1} F^*(x) \mathrm{d}t\ |\  F^* \in S(F) \Big\} = \Big(\int_{t_0}^{t_1} f_1(x) \mathrm{d}t, \dots, \int_{t_0}^{t_1} f_d(x) \mathrm{d}t \Big).
\end{equation*}
The integral above exists if $S(F) \neq \emptyset$, and in such cases, $F$ is referred to as Aumann integrable.

\begin{theorem}[See \cite{Leal}]\label{theorem2.11}
	Let $F = (f_1, \dots, f_d)$ be a measurable, integrable and bounded interval-vector-valued function on the interval $[t_0, t_1]$, where $F(x) = \big( [\underline{f}_1(x), \overline{f}_1(x)], \dots, [\underline{f}_d(x), \overline{f}_d(x)] \big)$, then
	\begin{enumerate}
		\item[(1)] $F$ is integrable, and
		\[
		\int_{t_0}^{t_1} F(x) \mathrm{d}t \in \mathcal{I}(\mathbb{R})^d,
		\]
		\item[(2)] $\underline{F} = (\underline{f}_1, \dots, \underline{f}_d)$ and $\overline{F} = (\overline{f}_1, \dots, \overline{f}_d)$ are two continuous vector-valued functions, and
		\[
		\int_{t_0}^{t_1} F(t) \mathrm{d}t = \Big( \Big[\int_{t_0}^{t_1} \underline{f}_1(t) \mathrm{d}t, \int_{t_0}^{t_1} \overline{f}_1(t) \mathrm{d}t\Big], \dots, \Big[\int_{t_0}^{t_1} \underline{f}_d(t) \mathrm{d}t, \int_{t_0}^{t_1} \overline{f}_d(t) \mathrm{d}t\Big] \Big).
		\]
	\end{enumerate}
\end{theorem}

\vskip 25pt
\section{\bf Existence of Open-loop Pareto-Nash equilibria}
\vskip 10pt

In this section, our focus lies on the analysis of multi-objective interval differential games and their (weighted) open-loop Pareto-Nash equilibria. We consider the control processes as strategies for the $n$ players, where each player's strategy space is defined by the permissible control set. Consequently, the $n$-player multi-objective interval differential game can be viewed as a specialized version of an $n$-player multi-objective game with interval payoffs. Based on the literature\cite{LiLi}, we established two existence theorems for open-loop Pareto-Nash equilibrium and weighted open-loop Pareto-Nash equilibrium in multi-objective interval differential games, respectively.

Consider the following dynamic system of state evolution in a multi-objective differential game
\begin{equation}\label{3.1}
	\dot{x}(t) = f(t, x(t), u_1(t), \dots, u_n(t)), \ x(t_0) = x_0,
\end{equation}
where $t_0$ and $x_0$ represent the initial time and initial state respectively, $x(t) \in \mathbb{R}^m$ represents the system state, $u_i(t) \in \mathbb{R}^{m_i}$ denotes the control function of Player $i$ ($i = 1, \dots, n$), and $f : [t_0, t_1] \times \mathbb{R}^m \times \mathbb{R}^{m_1} \times \dots \times \mathbb{R}^{m_n} \to \mathbb{R}^m$ is a vector-valued function.

The payoff function $\boldsymbol{J}_i = \{J^i_1, \dots, J^i_{d_i}\} : \mathbb{R}^m \times \mathbb{R}^{m_1} \times \dots \times \mathbb{R}^{m_n} \to \mathcal{I}(\mathbb{R})^{d_i}$ is defined for Player $i$ as follows: for each $(u_1, \dots, u_n) \in \mathbb{R}^{m_1} \times \dots \times \mathbb{R}^{m_n}$,
\begin{equation}\label{3.2}
	\boldsymbol{J}_i(u_1, \dots, u_n) = \boldsymbol{\psi}_i(x(t_1)) + \int_{t_0}^{t_1} \boldsymbol{L}_i(t, x(t), u_1(t), \dots, u_n(t)) \mathrm{d}t,
\end{equation}
where $t_1$ is a terminal time, $\boldsymbol{\psi}_i : \mathbb{R}^m \to \mathcal{I}(\mathbb{R})^{d_i}$ is the Player $i$'s terminal payoff function, and $\boldsymbol{L}_i : [t_0, t_1] \times \mathbb{R}^m \times \mathbb{R}^{m_1} \times \dots \times \mathbb{R}^{m_n} \to \mathcal{I}(\mathbb{R})^{d_i}$ is the Player $i$'s integral payoff function, $i = 1, \dots, n$.

\subsection{Open-loop Pareto-Nash equilibria}
Let $t_0, t_1 \in \mathbb{R}$, $0 \leq t_0 < t_1$, and
\[
C([t_0, t_1]; \mathbb{R}^{m_k}) = \big\{ u_k\ |\ u_k : [t_0, t_1] \to \mathbb{R}^{m_k}\ \mbox{is continuous}\big\}.
\]
For any $u_k \in C([t_0, t_1]; \mathbb{R}^{m_k})$, the norm of $u_k$ is defined as $\|u_k \|:= \max \limits_{t_0 \leq t \leq t_1} \|u_k(t)\|$. Assume that $U_k$ satisfies the following three conditions:

(A1) $U_k$ is a nonempty and closed subset of $C([t_0, t_1]; \mathbb{R}^{m_k})$;

(A1) $U_k$ is a uniformly bounded set, that is, there exists $M_k > 0$, such that for every $u_k \in U_k$, $\|u_k\| \leq M_k$;

(A2) $U_k$ is equicontinuous, that is, for each $\epsilon > 0$, there exists $\delta > 0$, such that for all $ t, t' \in [t_0, t_1]$ with $|t - t'| < \delta$, and for all $ u_k \in U_k$, we have $\|u_k(t) - u_k(t') \| < \epsilon$.

According to the renowned Ascoli-Arzela theorem, $U_k$ is a compact subset of $C([t_0, t_1]; \mathbb{R}^{m_k})$ and is considered as a set of permissible controls.

Let $M > 0$, $L > 0$, 
\[
D = \big\{(t,x, u_1, \dots, u_n) \in \mathbb{R}^{1 + m + m_1 + \dots + m_n}\ |\ t \in [t_0, t_1], x \in \mathbb{R}^m, \|u_k\| \leq M_k, k = 1, \dots, n\big\},
\]

(B1) $f$ is a bounded function on D if, there exists $M > 0$ such that 
\[
\sup \limits_{(t,x, u_1, \dots, u_n) \in D} \|f(t,x, u_1, \dots, u_n)\| \leq M.
\]

(B2) $f$ satisfies the Lipschitz condition with respect to the variable $x$ if, for all $(t,x_1, u_1, \dots, u_n)\in D$ and $(t,x_2, u_1, \dots, u_n) \in D$, it holds that
\[
\|f(t,x_1, u_1, \dots, u_n) - f(t,x_2, u_1, \dots, u_n) \| \leq  L \|x_1 - x_2\|.
\]
Let
\[
Y = \big\{ f : D \to \mathbb{R}^m |\ f\ \mbox{is continuous on}\ D, \mbox{and satisfies conditions (B1) and (B2)} \big\}.
\]
Therefore, the differential equation (\ref{3.1}) has a unique solution $x(t)$ for every $f \in Y$ and every $u_k \in U_k$ (where $k = 1, \dots, N$).

The set of permissible states is
\[
X = \big\{ x \in C^1([t_0, t_1]; \mathbb{R}^m)\ |\  \dot{x}(t) = f(t, x(t), u_1(t), \dots, u_n(t)), f \in Y, \ u_k \in U_k,\ x(t_0) = x_0 \big\}.
\]
\begin{definition}\label{definition3.1}
The control combination $(u_1^*, \dots, u_n^*) \in U_1 \times \dots \times U_n$ is an open-loop Pareto-Nash equilibrium of multi-objective interval differential game(\ref{3.1})-(\ref{3.2}), if for each $i$ and each $u_i \in U_i$, it holds that
\begin{equation*}
	\boldsymbol{J}_i(u_i^*, u_{-i}^*) \nprec \boldsymbol{J}_i(u_i, u_{-i}^*).
\end{equation*}
\end{definition}
Therefore, $(U_i, \boldsymbol{J}_i)_{i \in N}$ can be considered as a multi-objective game with interval payoffs in \cite{LiLi}. Based on Lemma \ref{lemma2.4} and Theorem \ref{theorem2.5}, we can deduce the following theorem.
\begin{theorem}\label{theorem3.2}
	For each $i \in N$, let $U_i$ be a nonempty convex subset of normed linear space $C([t_0, t_1]; \mathbb{R}^{m_i})$, and let $f \in Y$. If there exist $k_i \in \{1, \dots, d_i\}$ such that both of the following conditions are satisfied.
	\begin{enumerate}
	\item[(i)] $J^i_{k_i} : U_1 \times \dots \times U_n \to \mathcal{I}(\mathbb{R})$ is continuous;
	
	\item[(ii)] for each $u_{-i} \in U_{-i}$, $u_i \mapsto J^i_{k_i}(u_i, u_{-i})$ is generalized $\mathcal{I}(\mathbb{R_+})$-quasi-concave.
	\end{enumerate}
	Then the multi-objective interval differential game (\ref{3.1})-(\ref{3.2}) has at least one open-loop Pareto-Nash equilibrium solution. 
\end{theorem}
\begin{proof} 
	We construct the set-valued mapping $\Phi: U \rightrightarrows U$ as follows: for each $u = (u_i, u_{-i}) \in U$, 
	\[
	\Phi(u) = \prod_{i\in N} \Phi_i (u_{-i}),
	\]
	where
	\[
	\Phi_i (u_{-i}) = \{ v_i \in U_i\ |\ J_{k_i}^i(v_i, u_{-i}) \nprec J_{k_i}^i(u_i, u_{-i}),\ \forall\ u_i \in U_i \}.
	\]
	By Lemma \ref{lemma2.4}, we know that every $\Phi_i: U_{-i} \rightrightarrows U_i$ is nonempty, compact valued, and upper semicontinuous.
	
	Next, we will prove that the set $\Phi_i (u_{-i})$ is convex for any $u_{-i} \in U_{-i}$. Let's assume that $v_1, v_2 \in \Phi_i(u_{-i})$, and $\lambda \in [0, 1]$. It is evident that $v_1, v_2 \in U_i$. Furthermore,
	\[
	J_{k_i}^i(u_i, u_{-i}) \ominus_{gH} J_{k_i}^i(v_1, u_{-i}) \notin int \mathcal{I}(\mathbb{R_+})
	\]
	and
	\[
	J_{k_i}^i(u_i, u_{-i}) \ominus_{gH} J_{k_i}^i(v_2, u_{-i}) \notin int \mathcal{I}(\mathbb{R_+}).
	\]
	The convexity of $U_i$ implies that $\lambda v_1 + (1- \lambda)v_2 \in U_i$. Since $u_i \mapsto J_{k_i}^i(u_i, u_{-i})$ is generalized $\mathcal{I}(\mathbb{R_+})^{d_i}$-quasi-concave, then
	\[
	J_{k_i}^i(u_i, u_{-i})\ominus_{gH} J_{k_i}^i(\lambda v_1 + (1- \lambda)v_2, u_{-i}) \notin int \mathcal{I}(\mathbb{R_+}),
	\]
	that is, $\lambda v_1 + (1- \lambda)v_2 \in \Phi_i(u_{-i})$. Therefore, for every arbitrary $u_{-i} \in U_{-i}$, $\Phi_i(u_{-i})$ is a convex set.
	
	In summary, every $\Phi_i: U_{-i} \rightrightarrows U_i$ is an upper semicontinuous set-valued mapping with nonempty convex and compact values, and therefore $\Phi$ is as well. According to Theorem \ref{theorem2.5}, there exists $u^* \in U$ such that $u^* \in \Phi(u^*)$. That is, for any $i \in N$ and any $u_i \in U_i$,
	\[
	J_{k_i}^i(u_i, u_{-i}^*) \ominus_{gH} J_{k_i}^i(u_i^*, u_{-i}^*) \notin int \mathcal{I}(\mathbb{R_+}),
	\]
	furthermore,
	\[
	\boldsymbol{J}_i(u_i, u_{-i}^*) \ominus_{gH} \boldsymbol{J}_i(u_i^*, u_{-i}^*) \notin int \mathcal{I}(\mathbb{R_+})^{d_i}.
	\]
	Therefore, $x^*$ is an open-loop Pareto-Nash equilibrium of the game $(U_i, \boldsymbol{J}_i)_{i \in N}$.
\end{proof}
\subsection{Weighted open-loop Pareto-Nash equilibria}
Now, we consider the weight of each goal in the multi-objective interval differential game (\ref{3.1})-(\ref{3.2}). Let $\omega_i  = (\omega_1^i, \dots, \omega_{d_i}^i) \in int \mathbb{R}^{d_i}_+$ and $\Omega = (\omega_1, \dots, \omega_n)$, then Player $i$'s interval payoff function with weight $\Omega$ is defined as follows: for each $(u_1, \dots, u_n) \in U_1 \times \dots, \times U_n$,
\[
\left \langle \omega_i, \boldsymbol{J}_i (u_1, \dots, u_n) \right \rangle  = \sum_{k = 1}^{d_i}  \omega_k^i \cdot J_k^i(u_1, \dots, u_n).
\]
The multi-objective interval differential game with weight is then transformed into the following interval differential game.
\begin{equation}\label{3.3}
	\dot{x}(t) = f(t, x(t), u_1(t), \dots, u_n(t)), \ x(t_0) = x_0,
\end{equation}
\begin{equation}\label{3.4}
	\left \langle \omega_i, \boldsymbol{J}_i (u_1, \dots, u_n) \right \rangle  = \left \langle \omega_i, \boldsymbol{\psi}_i(x(t_1)) \right \rangle + \int_{t_0}^{t_1} \left \langle \omega_i, \boldsymbol{L}_i(t, x(t), u_1(t), \dots, u_n(t)) \right \rangle \mathrm{d}t.
\end{equation}
Therefore, an open-loop Pareto-Nash equilibrium of (\ref{3.3})-(\ref{3.4}) is referred to as a weighted open-loop Pareto-Nash equilibrium of (\ref{3.1})-(\ref{3.2}).

At this time, $(U_i, \left \langle \omega_i, \boldsymbol{J}_i \right \rangle)_{i \in N}$ can be regarded as a game with interval payoffs. 
Similar to the proof of Theorem \ref{theorem3.2}, we can derive the subsequent theorem.
\begin{theorem}\label{theorem3.3}
	For each $i \in N$, let $\omega_i  = (\omega_1^i, \dots, \omega_{d_i}^i) \in int \mathbb{R}^{d_i}_+$, and let $U_i$ be a nonempty convex subset of the normed linear space $C([t_0, t_1]; \mathbb{R}^{m_k})$. If $f \in Y$ and both of the following conditions are satisfied.
	\begin{enumerate}
		\item[(i)] $\sum_{k = 1}^{d_i}  \omega_k^i \cdot J_k^i : U_1 \times \dots \times U_n \to \mathcal{I}(\mathbb{R})$ is continuous, 
		\item[(ii)] for every $u_{-i} \in U_{-i}$, $u_i \mapsto \sum_{k = 1}^{d_i}  \omega_k^i \cdot J_k^i (u_i, u_{-i})$ is generalized $\mathcal{I}(\mathbb{R_+})$-quasi-concave.
	\end{enumerate}
    Then multi-objective interval differential game (\ref{3.1})-(\ref{3.2}) has at least one weighted open-loop Pareto-Nash equilibrium solution. 
\end{theorem}

\vskip 25pt
\section{\bf $n$-player interval differential games}
\vskip 10pt
In this section, our focus lies on the $n$-player interval differential game and determining the conditions under which its open-loop Pareto-Nash equilibrium exists. Initially, we construct an interval-valued Hamiltonian function for each player and apply the Pontryagin maximum principle to derive the necessary conditions for the existence of an open-loop Pareto-Nash equilibrium in the interval differential game. Subsequently, we utilize the concavity property of the endpoint-maximization Hamiltonian function to establish sufficient conditions for such an equilibrium to exist.

\subsection{Necessary conditions for the existence of open-loop Pareto-Nash equilibrium}
The interval differential game can be expressed through the state equation denoted as
\begin{equation}\label{4.1}
	\dot{x}(t) = f(t, x(t), u_1(t), \dots, u_n(t)), \ x(t_0) = x_0,
\end{equation}
while the payoff functions for Players 1 and 2 can be represented by 
\begin{equation}\label{4.2}
	J_i(u_1, \dots, u_n) = \psi_i(x(t_1)) + \int_{t_0}^{t_1} L_i(t, x(t), u_1(t), \dots, u_n(t)) \mathrm{d}t,
\end{equation}
where $f : [t_0, t_1] \times \mathbb{R}^m \times \mathbb{R}^{m_1} \times \dots \times \mathbb{R}^{m_n} \to \mathbb{R}^m$ is the state function, $\psi_i : \mathbb{R}^m \to \mathcal{I}(\mathbb{R})$ is Player $i$'s terminal payoff function, $L_i : [t_0, t_1] \times \mathbb{R}^m \times \mathbb{R}^{m_1} \times \dots \times \mathbb{R}^{m_n} \to \mathcal{I}(\mathbb{R})$ is Player $i$'s integral payoff function, $i = 1, \dots, n$.
\begin{definition}\label{definition4.1}
	The control combination $(u_1^*, \dots, u_n^*) \in U_1 \times \dots \times U_n$ is said to be an open-loop Pareto-Nash equilibrium of the interval differential game (\ref{4.1})-(\ref{4.2}), if for each $i \in N$ and each $u_i \in U_i$, it holds that
	\begin{equation*}
		\underline{J}_i(u_i^*, u_{-i}^*) \geq \underline{J}_i(u_i, u_{-i}^*)\ (\ \mbox{or}\ \overline{J}_i( u_i^*, u_{-i}^*) \geq \overline{J}_i(u_i, u_{-i}^*)\ ).
	\end{equation*}
    And $u_i^*$ is Player $i$'s equilibrium control, $J_i(u_i^*, u_{-i}^*)$ represents Player $i$'s Pareto optimal interval.
\end{definition}
To investigate the necessary conditions for the existence of an open-loop Pareto-Nash equilibrium, we define Player $i$'s interval-valued Hamilton function $H_i : [t_0, t_1] \times \mathbb{R}^m \times \mathbb{R}^{m_1} \times \dots \times \mathbb{R}^{m_n} \times \mathbb{R}^m \to \mathcal{I}(\mathbb{R})$ as follows.
\begin{equation}\label{4.3}
	H_i(t, x(t), u_1(t), \dots, u_n(t), p_i(t)) = L_i(t, x(t), u_1(t), \dots, u_n(t)) + p_i^T(t) f(t, x(t), u_1(t), \dots, u_n(t)),
\end{equation}
where $p_i(t) \in C^1([t_0, t_1]; \mathbb{R}^m)$ and
\[
L_i(t, x(t), u_1(t), \dots, u_n(t)) = \big[\underline{L}_i(t, x(t), u_1(t), \dots, u_n(t)), \overline{L}_i(t, x(t), u_1(t), \dots, u_n(t))\big].
\]
\begin{theorem}\label{theorem4.2}
	Assuming that for each $L_i : [t_0, t_1] \times \mathbb{R}^m \times \mathbb{R}^{m_1} \times \dots \times \mathbb{R}^{m_n} \to \mathcal{I}(\mathbb{R})$ ($i = 1, \dots, n$) has continuous $E$-derivative for $x(\cdot)$ and each $u_i(\cdot)$, and that $f : [t_0, t_1] \times \mathbb{R}^m \times \mathbb{R}^{m_1} \times \dots \times \mathbb{R}^{m_n} \to \mathbb{R}^m$ has continuous partial derivative for $x(\cdot)$ and each $u_i(\cdot)$. If $(u_1^*, \dots, u_n^*) \in U_1 \times \dots \times U_n$ is an open-loop Pareto-Nash equilibrium of interval differetial game (\ref{4.1})-(\ref{4.2}), with $x^*$ representing the corresponding state, then there exists $p_i(t) \in C^1([t_0, t_1]; \mathbb{R}^m)$ such that
	\begin{equation*}
		\left\{ \begin{array}{ll}
			& \dot{x}^*(t) = f (t, x^*(t), u_1^*(t), \dots, u_n^*(t)), \\
			& \dot{p}_1(t) = - \frac{\partial \hat{H}_1}{\partial x}(t, x^*(t), u_1^*(t), \dots, u_n^*(t), p_1(t)), \\
			& \ \dots\  \dots\  \dots\  \dots  \\
			& \dot{p}_n(t) = - \frac{\partial \hat{H}_n}{\partial x}(t, x^*(t), u_1^*(t), \dots, u_n^*(t), p_n(t)), \\
			& \frac{\partial \hat{H}_1}{\partial u_1}(t, x^*(t), u_1^*(t), \dots, u_n^*(t), p_1(t)) = 0,\\
			& \ \dots\  \dots\  \dots\  \dots  \\
			& \frac{\partial \hat{H}_n}{\partial u_n}(t, x^*(t), u_1^*(t), \dots, u_n^*(t), p_n(t)) = 0,\\
		\end{array} \right.
	\end{equation*}
	the initial and terminal conditions are given by
	\begin{equation*}
		\left\{ \begin{array}{ll}
			& p_1(t_1) = \nabla \hat{\psi}_1(x^*(t_1)), \\
			& \ \dots\  \dots\  \dots\  \\
			& p_n(t_1) = \nabla \hat{\psi}_n(x^*(t_1)), \\
			& x^*(t_0) = x_0,
		\end{array} \right.
	\end{equation*}
	where $\hat{H}_i = \hat{L}_i + p_i^T \cdot f$, $(\hat{L}_i, \hat{\psi}_i) = (\underline{L}_i, \underline{\psi}_i)$ or $(\overline{L}_i,\overline{\psi}_i)$, $i = 1, \dots, n$.
\end{theorem}
\begin{proof}
	Let $(u_1^*, \dots, u_n^*)$ be an open-loop Pareto-Nash equilibrium of the interval differential game (\ref{4.1})-(\ref{4.2}), and let $x^*$ denote the corresponding state. By Definition\ref{definition4.1}, for any fixed $i$ and every $u_i \in U_i$, it follows that
	\begin{equation}\label{4.4}
		\hat{J}_i (u_i^*, u_{-i}^*) \geq \hat{J}_i (u_i, u_{-i}^*), 
	\end{equation}
	where
	\begin{equation}\label{4.5}
		\hat{J}_i(u_1(t), \dots, u_n(t))  = \hat{\psi}_i(x(t_1)) + \int_{t_0}^{t_1} \hat{L}_i(t, x(t), u_1(t), \dots, u_n(t)) \mathrm{d}t.
	\end{equation}
	
	For any fixed $i$, according to (\ref{4.4}), $u_i^*$ is an optimal solution to the following optimal control problem.
	\begin{equation*}
		\begin{split}
			\max_{u_i \in U_i} &\ \hat{J}_i(u_i(t), u_{-i}^*(t))\\
			s. t.\  & \dot{x}(t) = f(t, x(t), u_i(t), u_{-i}^*(t)), \\
			&  x(t_0) = x_0.
		\end{split}
	\end{equation*} 
	According to the Pontryagin maximum principle, there exists $p_i(t) \in C^1([t_0, t_1]; \mathbb{R}^m)$ such that
	\[
	\hat{H}_i = \hat{L}_i + p_i^T \cdot f,
	\]
	and 
	\begin{equation*}
		\left\{ \begin{array}{ll}
			& \dot{x}^*(t) = f (t, x^*(t), u_1^*(t), \dots, u_n^*(t)), \\
			& \dot{p}_i(t) = - \frac{\partial \hat{H}_i}{\partial x}(t, x^*(t), u_1^*(t), \dots, u_n^*(t), p_i(t)), \\
			& \frac{\partial \hat{H}_i}{\partial u_i}(t, x^*(t), u_1^*(t), \dots, u_n^*(t), p_i(t)) = 0,
		\end{array} \right.
	\end{equation*}
	where the initial and terminal conditions are as follows:
	\begin{equation*}
		\left\{ \begin{array}{ll}
			& p_i(t_1) = \nabla \hat{\psi}_i(x^*(t_1)), \\
			& x^*(t_0) = x_0.
		\end{array} \right.
	\end{equation*}
	The conclusion holds, as we can infer from the arbitrariness of $i$.
\end{proof}

Subsequently, we employ the $gH$-partial derivative of the interval-valued Hamiltonian function $H_i$ to establish the requisite conditions for open-loop Pareto-Nash equilibrium in interval differential games. For ease of exposition, let $m = m_1 = \dots = m_n = 1$, although these findings also hold when $m, m_1, \dots, m_n > 1$.

Since $L_i$ is continuously $E$-differentiable, then $H_i$ is continuously $gH$-differentiable. Furthermore, the gradient of $H_i$ at $(x^*, u_1^*, \dots, u_n^*, p_i)$ is defined by
\begin{equation*}
	\begin{split}
		\nabla_{gH} H_i(x^*, u_1^*, \dots, u_n^*, p_i) & = \Big(\big(\frac{\partial H_i}{\partial x}\big)_{gH} (x^*, u_1^*, \dots, u_n^*, p_i), \big(\frac{\partial H_i}{\partial u_1}\big)_{gH} (x^*, u_1^*, \dots, u_n^*, p_i), \\
		& \dots \big(\frac{\partial H_i}{\partial u_n}\big)_{gH} (x^*, u_1^*, \dots, u_n^*, p_i), \big(\frac{\partial H_i}{\partial p_i}\big)_{gH} (x^*, u_1^*, \dots, u_n^*, p_i) \Big),
	\end{split}
\end{equation*}
where, the $gH$-partial derivative of $H_i$ is
\begin{equation*}
	\begin{split}
		\Big(\frac{\partial H_i}{\partial x}\Big)_{gH} (x^*, u_1^*, \dots, u_n^*, p_i) & = \big[\min\Big\{\frac{\partial \underline{H}_i}{\partial x} (x^*, u_1^*, \dots, u_n^*, p_i), \frac{\partial  \overline{H}_i}{\partial x} (x^*, u_1^*, \dots, u_n^*, p_i)\Big\}, \\
		&\ \ \ \ \ \max\Big\{\frac{\partial \underline{H}_i}{\partial x} (x^*, u_1^*, \dots, u_n^*, p_i), \frac{\partial \overline{H}_i}{\partial x} (x^*, u_1^*, \dots, u_n^*, p_i)\Big\} \Big].
	\end{split}
\end{equation*}
If $L_i$ is continuously $E$-differentiable and $f$ is continuously differentiable, then
\[
\Big(\frac{\partial H_i}{\partial x}\Big)_{gH} (x^*, u_1^*, \dots, u_n^*, p_i) = \Big(\frac{\partial L_i}{\partial x}\Big)_{gH} (x^*, u_1^*, \dots, u_n^*) + p_i^T(t) \cdot \frac{\partial f}{\partial x} (x^*, u_1^*, \dots, u_n^*),
\]
and 
\begin{equation}\label{4.6}
	\frac{\partial \hat{H}_i}{\partial x} (x^*, u_1^*, \dots, u_n^*, p_i) \in \Big(\frac{\partial H_i}{\partial x}\Big)_{gH} (x^*, u_1^*, \dots, u_n^*, p_i).
\end{equation}

\begin{theorem}\label{theorem4.3}
	Assuming that $L_i : [t_0, t_1] \times \mathbb{R}^m \times \mathbb{R}^{m_1} \times \dots \times \mathbb{R}^{m_n} \to \mathcal{I}(\mathbb{R})$ possesses continuous $E$-derivative for $x(\cdot)$ and each $u_i(\cdot)$, and that $f : [t_0, t_1] \times \mathbb{R}^m\times \mathbb{R}^{m_1}\times\dots\times 	\mathbb{R}^{m_n}\to	\mathbb{R}^m$ has continuous partial derivative for $x(\cdot)$ and each $u_i(\cdot)$, if $(u_1^*,\dots, u_n^*)$ represents an open-loop Pareto-Nash equilibrium of the interval differential game (\ref {4.1})-(\ref {4.2}), with corresponding state denoted by $x^*$, then there exist functions ${p_i(t) \in C^1 ([t _0 , t _1]; R ^ m)}$ (${i = 1,\dots, n}$), such that
	\begin{equation*}
		\left\{ \begin{array}{ll}
			& \dot{x}^*(t) = f (t, x^*(t), u_1^*(t), \dots, u_n^*(t)),\  x^*(t_0) = x_0\\
			& \dot{p}_1(t) \in - \Big(\frac{\partial H_1}{\partial x}\Big)_{gH}(t, x^*(t), u_1^*(t), \dots, u_n^*(t), p_i(t)),\\
			& \ \ \ \dots\  \dots\  \dots\  \dots\\
			& \dot{p}_n(t) \in - \Big(\frac{\partial H_n}{\partial x}\Big)_{gH}(t, x^*(t), u_1^*(t), \dots, u_n^*(t), p_i(t)),\\
			& 0 \in \Big(\frac{\partial H_1}{\partial u_1}\Big)_{gH} (t, x^*(t), u_1^*(t), \dots, u_n^*(t), p_i(t)),\\
			& \ \ \ \ \dots\  \dots\  \dots\  \dots\\
			& 0 \in \Big(\frac{\partial H_n}{\partial u_n}\Big)_{gH} (t, x^*(t), u_1^*(t), \dots, u_n^*(t), p_i(t)),\\
			& p_i(t_1) \in \nabla_{gH} \psi_i(x^*(t_1)), \ i = 1, \dots, n.\\
		\end{array} \right.
	\end{equation*}
\end{theorem}
\begin{proof}
	Assuming that $(u_1^*, \dots, u_n^*)$ constitutes an open-loop Pareto-Nash equilibrium of (\ref{4.1})-(\ref{4.2}), according to Theorem \ref{theorem4.2}, there exists $p_i(t) \in C^1([t_0, t_1]; \mathbb{R}^m)$ such that
	\[
	\dot{p}_1(t) = - \frac{\partial \hat{H}_1}{\partial x}(t, x^*(t), u_1^*(t), \dots, u_n^*(t), p_1(t)),
	\]
	where $\hat{H}_i = \hat{L}_i + p_i^T f$. Combining formula (\ref{4.6}), we can get
	\[
	\dot{p}_i(t) \in - \Big(\frac{\partial H_i}{\partial x}\Big)_{gH}(t, x^*(t), u_1^*(t), \dots, u_n^*(t), p_1(t)).
	\]
	Similarly, other conditions can be acquired.
\end{proof}

\subsection{Sufficient conditions for the existence of open-loop Pareto-Nash equilibrium}

In the following section, we establish sufficient conditions for the existence of an open-loop Pareto-Nash equilibrium in interval differential games (\ref{4.1})-(\ref{4.2}). To accomplish this, we first employ Player $i$'s real-valued Hamilton function $\hat{H}_i = \hat{L}_i + p_i^T \cdot f$ to construct the maximization Hamilton function $\hat{H}_i^*$.
\begin{equation}\label{4.7}
	\hat{H}_i^*(t, x(t), u_{-i}(t),p_i(t)) = \max \big\{\hat{H}_i(t, x(t), u_i(t), u_{-i}(t),p_i(t))\ |\ u_i \in U_i \big\}.
\end{equation}
\begin{theorem}\label{theorem4.4}
	Let $X$ be a convex set, let $(u_1^*, \dots, u_n^*) \in U_1 \times \dots \times U_n$, and let $x^*(\cdot)$ be the corresponding state. For each player $i$, suppose the terminal payoff function $\psi_i$ is continuously $E$-differentiable and $LU$-concave. If there exists an absolutely continuous function $p_i : [t_0, t_1] \to \mathbb{R}^m$ satisfying 
	\[
	\hat{H}_i(t, x^*(t), u_i^*(t), u_{-i}^*(t),p_i(t)) = \hat{H}_i^*(t, x^*(t), u_{-i}^*(t),p_i(t)),
	\]
	\[
	\dot{p}_i(t) = - \frac{\partial \hat{H}_i^*}{\partial x}(t, x^*(t), u_{-i}^*(t), p_i(t)),
	\]
	\[
	p_i(t_1) = \nabla \hat{\psi}_i(x^*(t_1)),
	\]
	and if $\hat{H}_i^*(t, x(t), u_{-i}(t),p_i(t))$ is concave and continuously differentiable with respect to $x(\cdot)$, then $(u_1^*, \dots, u_n^*)$ constitutes an open-loop Pareto-Nash equilibrium of (\ref{4.1})-(\ref{4.2}).
\end{theorem}
\begin{proof}
	For each $i$, let $u_i$ represent any admissible control for Player $i$, and let $x(\cdot)$ denote the corresponding state resulting from $(u_i, u_{-i}^*)$. Our objective is to establish that $\hat{J}_i(u_i^*, u_{-i}^*) \geq \hat{J}_i(u_i, u_{-i}^*)$.
	
	Given the concavity and continuous differentiability of $\hat{H}_i^*(t, x(t), u_{-i}(t),p_i(t))$ with respect to $x(\cdot)$, it follows that 
	\begin{equation}\label{4.8}
		\begin{split}
			&\ \  \ \hat{H}_i^*(t, x^*(t), u_{-i}^*(t),p_i(t)) - \hat{H}_i^*(t, x(t), u_{-i}^*(t),p_i(t))\\
			& \geq \frac{\partial \hat{H}_i^*}{\partial x} (t, x^*(t), u_{-i}^*(t),p_i(t)) [x^*(t) - x(t)]\\
			& = - \dot{p}_i^T(t) [x^*(t) - x(t)].
		\end{split}
	\end{equation}

    Due to the continuous $E$-differentiability and $LU$-concavity of $\psi_i$, we can conclude that $\hat{\psi}_i$ is also continuously differentiable and concave. Therefore,
    \[
    \hat{\psi}_i(x^*(t_1)) - \hat{\psi}_i(x(t_1)) \geq \nabla \hat{\psi}_i(x^*(t_1)) [x^*(t_1) - x(t_1)] = p_i^T(t_1) [x^*(t_1) - x(t_1)],
    \]
    that is,
    \begin{equation}\label{4.9}
    \hat{\psi}_i(x^*(t_1)) - \hat{\psi}_i(x(t_1)) + p_i^T(t_1) [x(t_1) - x^*(t_1)] \geq 0.
    \end{equation}

	Based on (\ref{4.1}), (\ref{4.5}), and (\ref{4.7})-(\ref{4.9}), we can conclude that
	\begin{equation*}
		\begin{split}
			&\ \ \  \hat{J}_i(u_i^*, u_{-i}^*) - \hat{J}_i(u_i, u_{-i}^*) \\
			& = \hat{\psi}_i(x^*(t_1)) + \int_{t_0}^{t_1} \hat{L}_i(t, x^*(t), u_i^*(t), u_{-i}^*(t)) \mathrm{d}t - \hat{\psi}_i(x(t_1)) - \int_{t_0}^{t_1} \hat{L}_i(t, x(t), u_i(t), u_{-i}^*(t)) \mathrm{d}t\\
			& = \hat{\psi}_i(x^*(t_1)) + \int_{t_0}^{t_1} \big[\hat{H}_i(t, x^*(t), u_i^*(t), u_{-i}^*(t),p_i(t)) - p_i^T(t) \dot{x}^*(t)\big] \mathrm{d}t\\
			&\ \ - \hat{\psi}_i(x(t_1)) - \int_{t_0}^{t_1} \big[\hat{H}_i(t, x(t), u_i(t), u_{-i}^*(t),p_i(t)) - p_i^T(t) \dot{x}(t)\big] \mathrm{d}t \\
			& \geq \hat{\psi}_i(x^*(t_1)) + \int_{t_0}^{t_1} \big[\hat{H}_i^*(t, x^*(t), u_{-i}^*(t),p_i(t)) - p_i^T(t) \dot{x}^*(t)\big] \mathrm{d}t\\
			&\ \ - \hat{\psi}_i(x(t_1)) - \int_{t_0}^{t_1} \big[\hat{H}_i^*(t, x(t), u_{-i}^*(t),p_i(t)) - p_i^T(t) \dot{x}(t)\big] \mathrm{d}t\\
			& \geq \hat{\psi}_i(x^*(t_1)) - \hat{\psi}_i(x(t_1)) + \int_{t_0}^{t_1} \big\{\dot{p}_i^T(t) [x(t) - x^*(t)] + p_i^T(t) [\dot{x}(t) - \dot{x}^*(t)] \big\} \mathrm{d}t\\
			& = \int_{t_0}^{t_1} \frac{\mathrm{d}}{\mathrm{d}t} \{p_i^T(t) [x(t) - x^*(t)]\} \mathrm{d}t + \hat{\psi}_i(x^*(t_1)) - \hat{\psi}_i(x(t_1))\\
			& = p_i^T(t_1) [x(t_1) - x^*(t_1)] + \hat{\psi}_i(x^*(t_1)) - \hat{\psi}_i(x(t_1)) \geq 0.
		\end{split}
	\end{equation*}
	
	 Therefore, $(u_1^*, \dots, u_n^*)$ constitutes an open-loop Pareto-Nash equilibrium of (\ref{4.1})-(\ref{4.2}).
\end{proof}

\vskip 25pt
\section{\bf Two-player linear quadratic interval differential games}
\vskip 10pt
In this section, our focus lies on the linear quadratic interval differential game involving two players and the computation of its open-loop Pareto-Nash equilibrium. Firstly, we utilize Definition 4.1 to transform the linear quadratic interval differential game into four conventional linear quadratic differential games. Secondly, by employing Theorem 4.2 and solving a system of linear equations, we can derive a specific methodology for calculating the open-loop Pareto-Nash equilibrium of the game. Furthermore, we present a numerical example to illustrate the practical application of this computational approach.

The linear quadratic interval differential game can be expressed through the state equation denoted as
\begin{equation}\label{5.1}
	\dot{x}(t) = a x(t) + b u_1(t) + c u_2(t), \ x(t_0) = x_0,
\end{equation}
while the payoff functions for Players 1 and 2 can be represented by
\begin{equation}\label{5.2}
	J_1(u_1, u_2) = \frac{1}{2}\int_{t_0}^{t_1} \big\{[\underline{g}_1, \overline{g}_1]x^2(t) + [\underline{g}_2, \overline{g}_2] u_1^2(t) \big\} \mathrm{d}t,
\end{equation}
\begin{equation}\label{5.3}
	J_2(u_1, u_2) = \frac{1}{2}\int_{t_0}^{t_1} \big\{[\underline{h}_1, \overline{h}_1]x^2(t) + [\underline{h}_2, \overline{h}_2] u_2^2(t) \big\} \mathrm{d}t,
\end{equation}
where $a, b, c$ are both constants, $x(t) \in \mathbb{R}$, and $[\underline{g}_i, \overline{g}_i], [\underline{h}_i, \overline{h}_i] \in \mathcal{I}(\mathbb{R})$, $u_i(t) \in \mathbb{R}$, $i = 1, 2$.

Let $(u_1(\cdot), u_2(\cdot))$ be an open-loop Pareto-Nash equilibrium of linear quadratic interval differential game (\ref{5.1})-(\ref{5.3}), then the Hamiltonian functions of Players 1 and 2 are expressed as
\[
\mathcal{H}_1(t, x(t), u_1(t), u_2(t), p_1(t)) = \frac{1}{2} \big\{[\underline{g}_1, \overline{g}_1]x^2(t) + [\underline{g}_2, \overline{g}_2] u_1^2(t) \big\} + p_1(t) [a x(t) + b u_1(t) + c u_2(t)]
\]
\[
\mathcal{H}_2(t, x(t), u_1(t), u_2(t), p_2(t)) = \frac{1}{2} \big\{[\underline{h}_1, \overline{h}_1]x^2(t) + [\underline{h}_2, \overline{h}_2] u_2^2(t) \big\} + p_2(t) [a x(t) + b u_1(t) + c u_2(t)],
\]
where $p_i(t) (i = 1, 2)$ is the Hamiltonian adjoint variable of Player $i$.

According to Definition \ref{definition4.1}, $(u_1, u_2)$ constitutes an open-loop Nash equilibrium in one of the following four linear quadratic differential games.

\begin{equation}\label{5.4}
	\left\{ \begin{array}{ll}
		& \dot{x}(t) = a x(t) + b u_1(t) + c u_2(t), x(t_0) = x_0,\\ 
		& \underline{J_1}(u_1, u_2) = \frac{1}{2}\int_{t_0}^{t_1} [\underline{g}_1 x^2(t) + \underline{g}_2 u_1^2(t) ] \mathrm{d}t,\\
		& \underline{J_2}(u_1, u_2) = \frac{1}{2}\int_{t_0}^{t_1} [\underline{h}_1 x^2(t) + \underline{h}_2 u_2^2(t)] \mathrm{d}t.
	\end{array} \right.
\end{equation}

\begin{equation}\label{5.5}
	\left\{ \begin{array}{ll}
		& \dot{x}(t) = a x(t) + b u_1(t) + c u_2(t), x(t_0) = x_0, \\
		& \underline{J_1}(u_1, u_2) = \frac{1}{2}\int_{t_0}^{t_1} [\underline{g}_1 x^2(t) + \underline{g}_2 u_1^2(t) ] \mathrm{d}t,\\
		& \overline{J_2}(u_1, u_2) = \frac{1}{2}\int_{t_0}^{t_1} [\overline{h}_1 x^2(t) + \overline{h}_2 u_2^2(t)] \mathrm{d}t.
	\end{array} \right.
\end{equation}

\begin{equation}\label{5.6}
	\left\{ \begin{array}{ll}
		& \dot{x}(t) = a x(t) + b u_1(t) + c u_2(t), x(t_0) = x_0,\\
		& \overline{J_1}(u_1, u_2) = \frac{1}{2}\int_{t_0}^{t_1} [\overline{g}_1 x^2(t) + \overline{g}_2 u_1^2(t) ] \mathrm{d}t,\\
		& \underline{J_2}(u_1, u_2) = \frac{1}{2}\int_{t_0}^{t_1} [\underline{h}_1 x^2(t) + \underline{h}_2 u_2^2(t)] \mathrm{d}t.
	\end{array} \right.
\end{equation}

\begin{equation}\label{5.7}
	\left\{ \begin{array}{ll}
		& \dot{x}(t) = a x(t) + b u_1(t) + c u_2(t), x(t_0) = x_0,\\
		& \overline{J_1}(u_1, u_2) = \frac{1}{2}\int_{t_0}^{t_1} [\overline{g}_1 x^2(t) + \overline{g}_2 u_1^2(t) ] \mathrm{d}t,\\
		& \overline{J_2}(u_1, u_2) = \frac{1}{2}\int_{t_0}^{t_1} [\overline{h}_1 x^2(t) + \overline{h}_2 u_2^2(t)] \mathrm{d}t.
	\end{array} \right.
\end{equation}
 
Without loss of generality, let us assume that $(u_1(\cdot), u_2(\cdot))$ represents the open-loop Nash equilibrium of the linear quadratic differential game (\ref{5.4}). In this scenario, the real-valued Hamiltonian functions for players 1 and 2 can be expressed as
\[
\underline{\mathcal{H}_1}(t, x(t), u_1(t), u_2(t), p_1(t)) = \frac{1}{2} [\underline{g}_1 x^2(t) + \underline{g}_2 u_1^2(t)]+ p_1(t) (a x(t) + b u_1(t) + c u_2(t)),
\]
\[
\underline{\mathcal{H}_2}(t, x(t), u_1(t), u_2(t), p_2(t)) = \frac{1}{2} [\underline{h}_1 x^2(t) + \underline{h}_2 u_2^2(t)] + p_2(t) (a x(t) + b u_1(t) + c u_2(t)).
\]

In the absence of control and state constraints, the relationship between the control variable $u_i(t)$ and the Hamiltonian adjoint variable $p_i(t)$ can be described as
\begin{equation}\label{5.8}
	u_1(t) = - \frac{b}{\underline{g}_2} p_1(t), \  u_2(t) = - \frac{c}{\underline{h}_2} p_2(t),
\end{equation}
and the adjoint equation is provided by
\begin{equation*}
	\begin{split}
		& \dot{p}_1(t) = - \underline{g}_1 x(t) - a p_1(t),\\
		& \dot{p}_2(t) = - \underline{h}_1 x(t) - a p_2(t).
	\end{split}
\end{equation*}

By substituting equation (\ref{5.8}) into the state equation in (\ref{5.4}) and integrating it with the adjoint equation, we can derive the following system of equations.
\begin{equation}\label{5.9}
	\left\{ \begin{array}{ll}
		& \dot{x}(t) = a x(t) - \frac{b^2}{\underline{g}_2} p_1(t) - \frac{c^2}{\underline{h}_2} p_2(t),\\
		& \dot{p}_1(t) = - \underline{g}_1 x(t) - a p_1(t),\\
		& \dot{p}_2(t) = - \underline{h}_1 x(t) - a p_2(t),\\
		& p_1(t_1) = p_2(t_1) = 0, \ x(t_0) = x_0.
	\end{array} \right.
\end{equation}

Let $y(t) = (x(t), p_1(t), p_2(t))^T$ and
\begin{equation*}
	A =
	\begin{pmatrix}
		a & - \frac{b^2}{\underline{g}_2} & - \frac{c^2}{\underline{h}_2}\\
		- \underline{g}_1 &  - a  &  0 \\
		- \underline{h}_1 &  0  &  - a  \\
	\end{pmatrix}.
\end{equation*}
The system of equations (5.9) can be represented in the matrix form presented below.
\begin{equation}\label{5.10}
	\dot{y}(t) = A y(t).
\end{equation}
Therefore, the open-loop Nash equilibrium of (\ref{5.4}) is fully represented by (\ref{5.8}) and the system (\ref{5.10}). To find the open-loop Nash equilibrium, we need to solve a system of first-order linear differential equations. 

Next, we exclusively consider finite-time problems where $t_1 < \infty$. In this scenario, we are provided with the initial condition $x(t_0) = x_0$ and the terminal condition $p_1(t_1) = p_2(t_1) = 0$. Consequently, the system (\ref{5.10}) represents a two-point boundary value problem.
To solve (\ref{5.10}), we will proceed through three sequential steps.

Firstly, we compute the eigenvalues and corresponding eigenvectors of the matrix A. The determinant of the matrix $A$ is $\det A = a^3 + M a$, where
\[
M = b^2 \frac{\underline{g}_1}{\underline{g}_2} + c^2 \frac{\underline{h}_1}{\underline{h}_2}.
\]
The characteristic equation of matrix $A$ is given by 
\[
\det(sE - A) = (s + a)^2 (s - a) - (s + a) M = 0, 
\]
and it possesses three distinct eigenvalues, namely $s_1 = \sqrt{a^2 + M}$, $s_2 = - \sqrt{a^2 + M}$, and $s_3 = - a$. There are three distinct scenarios based on the sign of each eigenvalue.

(i) $s_1 > 0$, $s_2 < 0$, $s_3 > 0$.  This situation occurs when $\det A < 0$.

(ii) $s_1 > 0$, $s_2 < 0$, $s_3 < 0$. This situation occurs when $\det A > 0$.

(iii) 
$s_1 > 0$, $s_2 < 0$, $s_3 = 0$. This situation occurs when $\det A = 0$.

Secondly, we utilize the eigenvalues and eigenvectors to deduce the fundamental solution of (\ref{5.10}). Assuming solely a stable uncontrolled system $(a < 0)$, which corresponds to case (i) mentioned above. Consequently, two eigenvalues of $A$ are positive while one is negative, indicating the presence of a saddle point in the regular system. The general solution for the regular system (\ref{5.10}) can be formulated as
\begin{equation}\label{5.11}
	y(t) = e^{At} y(t_0) = W e^{\Lambda t} \alpha,
\end{equation}
where $w_i = (w_{1i}, w_{2i}, w_{3i})^T$ is the eigenvector of $s_i$, and
\begin{align*}
	W =  \begin{pmatrix}
		w_{11} & w_{12} & w_{13}\\
		w_{21} & w_{22} & w_{23}\\
		w_{31} & w_{32} & w_{33}
	\end{pmatrix},
	&\ \  e^{\Lambda t} =  \begin{pmatrix}
		e^{s_1 t} & 0 & 0\\
		0 & e^{s_2 t} & 0\\
		0 & 0 & e^{s_3 t}
	\end{pmatrix}.
\end{align*}

Finally, we compute a specific solution that satisfies the boundary conditions. Assuming the validity of our assumption, the three eigenvalues are distinct, ensuring linear independence of the eigenvectors. By employing equation (\ref{5.11}), we can express the three boundary conditions as follows.
\begin{equation*}
	\begin{split}
		x_0 & = \alpha_1 w_{11} e^{s_1 t_0} + \alpha_2 w_{12} e^{s_2 t_0} +\alpha_3 w_{13} e^{s_3 t_0},\\
		0 & = \alpha_1 w_{21} e^{s_1 t_1} + \alpha_2 w_{22} e^{s_2 t_1} +\alpha_3 w_{23} e^{s_3 t_1},\\
		0 & = \alpha_1 w_{31} e^{s_1 t_1} + \alpha_2 w_{32} e^{s_2 t_1} +\alpha_3 w_{33} e^{s_3 t_1}.
	\end{split}
\end{equation*}
The parameters $\alpha_1, \alpha_2, \alpha_3$ can be readily obtained by solving the aforementioned linear equations. Subsequently, utilizing these parameters in conjunction with equations (\ref{5.8}) and (\ref{5.11}), the open-loop Nash equilibrium is fully determined.

By employing a similar methodology, we can derive three additional open-loop Pareto-Nash equilibria for the linear quadratic interval differential game (\ref{5.1})-(\ref{5.3}).
\begin{example}
	The linear quadratic interval differential game can be formulated as
	\begin{equation*}
		\dot{x}(t) = 2 x(t) + u_1(t) + u_2(t), \ x(0) = 1,
	\end{equation*}
	\begin{equation*}
		J_1(u_1, u_2) = \frac{1}{2}\int_0^3 \big\{[0.9, 1.2]x^2(t) + [0.3, 0.6] u_1^2(t) \big\} \mathrm{d}t,
	\end{equation*}
	\begin{equation*}
		J_2(u_1, u_2) = \frac{1}{2}\int_0^3 \big\{[0.8, 1.5]x^2(t) + [0.4, 0.5] u_2^2(t) \big\} \mathrm{d}t.
	\end{equation*}

The open-loop Pareto-Nash equilibrium of the aforementioned interval differential game can be transformed into calculating the open-loop Nash equilibria of the following four linear quadratic differential games.
\begin{equation}\label{5.12}
	\left\{ \begin{array}{ll}
		& \dot{x}(t) = 2 x(t) + u_1(t) + u_2(t), x(0) = 1,\\
		& \underline{J_1}(u_1, u_2) = \frac{1}{2}\int_0^3 [0.9 x^2(t) + 0.3 u_1^2(t) ] \mathrm{d}t,\\
		& \underline{J_2}(u_1, u_2) = \frac{1}{2}\int_0^3 [0.8 x^2(t) + 0.4 u_2^2(t)] \mathrm{d}t.
	\end{array} \right.
\end{equation}

\begin{equation}\label{5.13}
	\left\{ \begin{array}{ll}
		& \dot{x}(t) = 2 x(t) + u_1(t) + u_2(t), x(0) = 1,\\
		& \underline{J_1}(u_1, u_2) = \frac{1}{2}\int_0^3 [0.9 x^2(t) + 0.3 u_1^2(t) ] \mathrm{d}t,\\
		& \overline{J_2}(u_1, u_2) = \frac{1}{2}\int_0^3 [1.5 x^2(t) + 0.5 u_2^2(t)] \mathrm{d}t.
	\end{array} \right.
\end{equation}

\begin{equation}\label{5.14}
	\left\{ \begin{array}{ll}
		& \dot{x}(t) = 2 x(t) + u_1(t) + u_2(t), x(0) = 1,\\
		& \overline{J_1}(u_1, u_2) = \frac{1}{2}\int_0^3 [1.2 x^2(t) + 0.6 u_1^2(t)] \mathrm{d}t,\\
		& \underline{J_2}(u_1, u_2) = \frac{1}{2}\int_0^3 [0.8 x^2(t) + 0.4 u_2^2(t)] \mathrm{d}t.
	\end{array} \right.
\end{equation}

\begin{equation}\label{5.15}
	\left\{ \begin{array}{ll}
		& \dot{x}(t) = 2 x(t) + u_1(t) + u_2(t), x(0) = 1,\\
		& \overline{J_1}(u_1, u_2) = \frac{1}{2}\int_0^3 [1.2 x^2(t) + 0.6 u_1^2(t)] \mathrm{d}t,\\
		& \overline{J_2}(u_1, u_2) = \frac{1}{2}\int_0^3 [1.5 x^2(t) + 0.5 u_2^2(t)] \mathrm{d}t.
	\end{array} \right.
\end{equation}
\begin{figure}
	\centering
	\begin{minipage}[t]{0.49 \linewidth}
		\centering
		\includegraphics[width=0.85
		\linewidth]{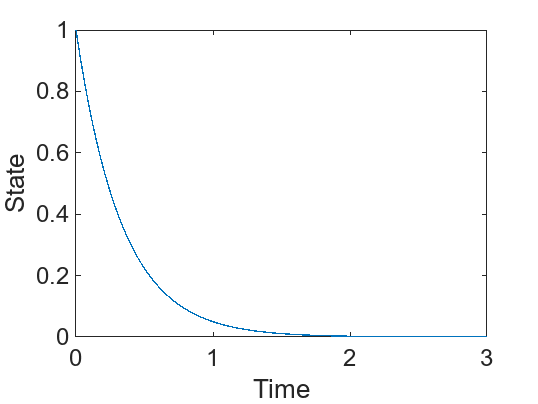}
		\caption*{\footnotesize{(a)\ The corresponding status}}
	\end{minipage}
	\begin{minipage}[t]{0.49 \linewidth}
		\centering
		\includegraphics[width=0.85\linewidth]{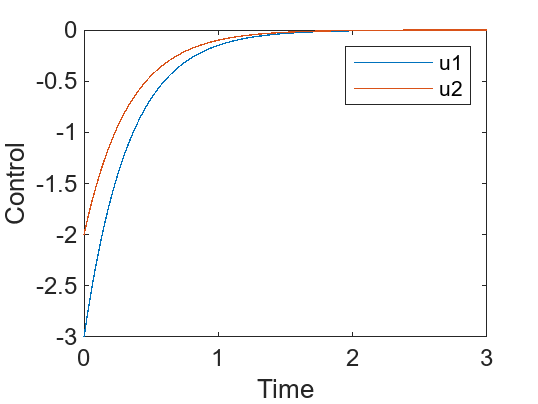}
		\caption*{\footnotesize{(b)\ Open-loop Pareto-Nash equilibrium solution}}
	\end{minipage}
	\centering
	\begin{minipage}[t]{0.49 \linewidth}
		\centering
		\includegraphics[width=0.85\linewidth]{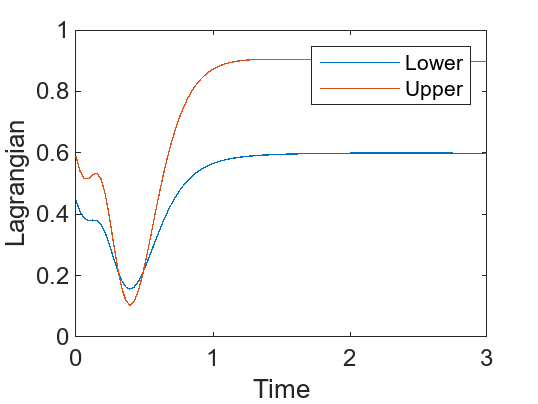}
		\caption*{\footnotesize{(c)\ The interval Lagrange function of Player 1}}
	\end{minipage}
	\begin{minipage}[t]{0.49 \linewidth}
		\centering
		\includegraphics[width=0.85\linewidth]{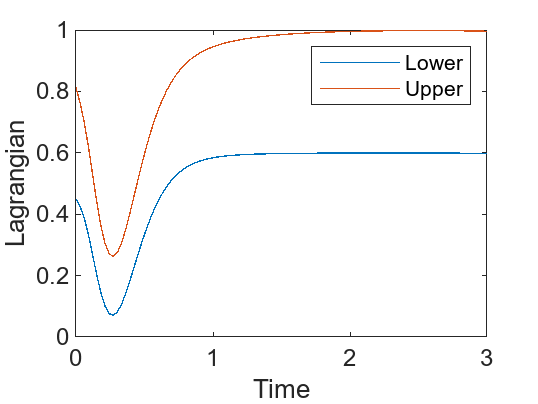}
		\caption*{\footnotesize{(d)\ The interval Lagrange function of Player 2}}
	\end{minipage}
	\caption{The linear quadratic differential game (\ref{5.12})}
	\label{fig:5.12}
\end{figure}

By calculation, we can obtain the open-loop Pareto-Nash equilibrium of the linear quadratic differential game (\ref{5.12}) as
\[
u^*_1(t) = \frac{3}{5 + e^{18}} (e^{3t} - e^{18 - 3t})
\]
and
\[ 
\ u^*_2(t) = \frac{2}{5 + e^{18}} (e^{3t} - e^{18 - 3t}).
\]
The corresponding status is
\[
x^*(t) = \frac{1}{5 + e^{18}} (5 e^{3t} + e^{18 - 3t}).
\]
At the current time, the interval Lagrange functions of Players 1 and 2 are as follows:
\[
L_1(x^*, u_1^*) =  \Big[\frac{1}{2} \big(0.9 x^{*2}(t) + 0.3 u_1^{*2}(t)\big), \frac{1}{2} \big(1.2 x^{*2}(t) + 0.6 u_1^{*2}(t)\big) \Big],
\]
\[
L_2(x^*, u_2^*) = \big[\frac{1}{2} \big(0.8 x^{*2}(t) + 0.4 u_2^{*2}(t)\big), \frac{1}{2} \big(1.5 x^{*2}(t) + 0.5 u_2^{*2}(t)\big) \big].
\]

The Pareto optimal payment interval of the two players can be obtained by calculating the integral.
\begin{equation*}
	\begin{split}
		J_1(u_1^*, u_2^*) & = \Big[\int_0^3 \frac{1}{2} \big\{0.9 x^{*2}(t) + 0.3 u_1^{*2}(t) \big\} \mathrm{d}t, \int_0^3 \frac{1}{2} \big\{1.2 x^{*2}(t) + 0.6 u_1^{*2}(t) \big\} \mathrm{d}t\Big]\\
		& = \frac{1}{(5 + e^{18})^2} \Big[ \frac{3}{10} e^{36} + \frac{36}{5} e^{18} - \frac{21}{10}, \frac{11}{20} e^{36} + \frac{21}{5} e^{18} - \frac{59}{20} \Big],
	\end{split}
\end{equation*}

\begin{equation*}
	\begin{split}
		J_2(u_1^*, u_2^*) & = \Big[\int_0^3 \frac{1}{2} \big\{0.8 x^{*2}(t) + 0.4 u_2^{*2}(t) \big\} \mathrm{d}t, \int_0^3 \frac{1}{2} \big\{1.5 x^{*2}(t) + 0.5 u_2^{*2}(t) \big\} \mathrm{d}t\Big]\\
		& = \frac{1}{(5 + e^{18})^2} \Big[\frac{1}{5} e^{36} + \frac{44}{5} e^{18} - \frac{9}{5}, \frac{7}{24} e^{36} + \frac{39}{2} e^{18} - \frac{79}{24} \Big].
	\end{split}
\end{equation*}

The four subfigures (a)-(d) in Figure \ref{fig:5.12} respectively illustrate the corresponding state curve of the linear quadratic differential game (5.12), the control curve of the open-loop Pareto-Nash equilibrium, and the curves of the interval Lagrange functions for both players.

\begin{figure}[htbp]
	\centering
	\begin{minipage}[t]{0.49 \linewidth}
		\centering
		\includegraphics[width=0.85\linewidth]{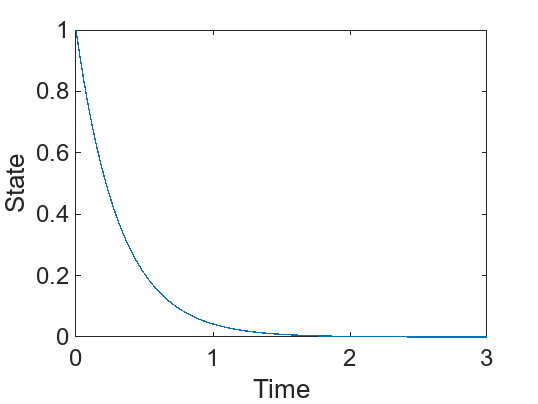}
		\caption*{\footnotesize{(a)\ Corresponding status}}
	\end{minipage}
	\begin{minipage}[t]{0.49 \linewidth}
		\centering
		\includegraphics[width=0.85\linewidth]{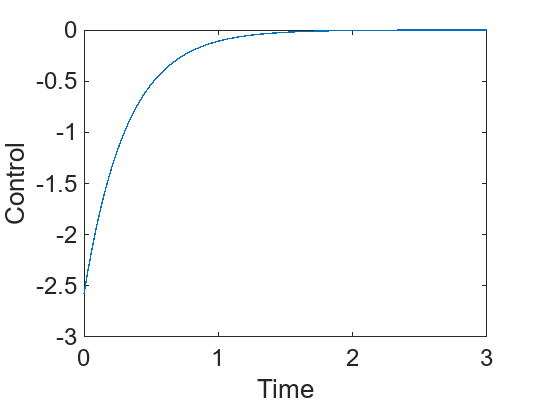}
		\caption*{\footnotesize{(b)\ Open-loop Pareto-Nash equilibrium solution}}
	\end{minipage}
	\centering
	\begin{minipage}[t]{0.49 \linewidth}
		\centering
		\includegraphics[width=0.85\linewidth]{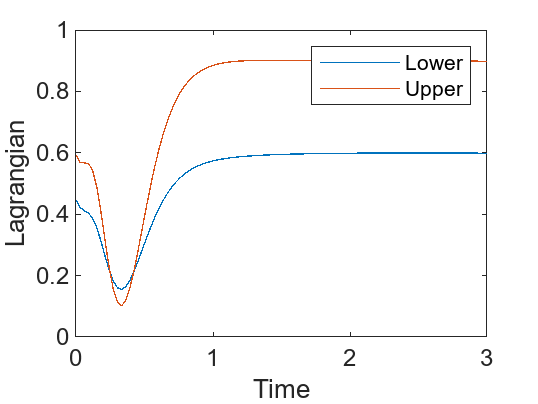}
		\caption*{\footnotesize{(c)\ Interval-valued Lagrange function of Player 1}}
	\end{minipage}
	\begin{minipage}[t]{0.49 \linewidth}
		\centering
		\includegraphics[width=0.85\linewidth]{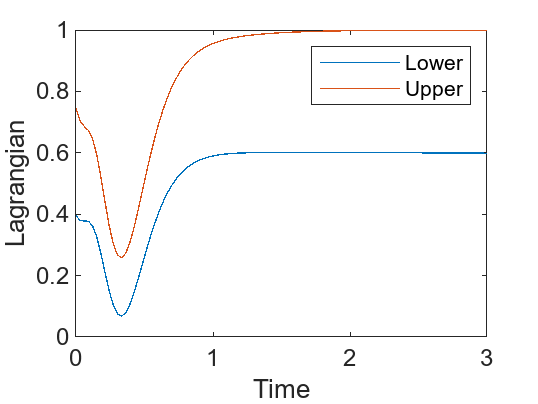}
		\caption*{\footnotesize{(d)\ Interval-valued Lagrange function of Player 2}}
	\end{minipage}
	\caption{Linear quadratic differential game (\ref{5.13})} 
	\label{fig:5.13}
\end{figure}

Similarly, the open-loop Pareto-Nash equilibrium of the linear quadratic differential game (\ref{5.13}) is solved as
\[
u^*_1(t) = u^*_2(t) = \frac{3}{K} \Big(e^{\sqrt{10}t} - e^{6\sqrt{10} - \sqrt{10}t}\Big),
\]
and the corresponding status is
\[
x^*(t) = \frac{1}{K} \Big[\big(\sqrt{10} + 2\big) e^{\sqrt{10}t} + \big(\sqrt{10} - 2\big) e^{6\sqrt{10} - \sqrt{10}t}\Big],
\]
where $K = \big(\sqrt{10} + 2\big) + \big(\sqrt{10} - 2\big) e^{6\sqrt{10}}$.

The payment intervals corresponding to Players 1 and 2 are as follows.
\begin{equation*}
		J_1(u_1^*, u_2^*) = \frac{3}{10} \frac{1}{K^2} \Big[33 e^{6\sqrt{10}} + \frac{3 \big(17 - 4\sqrt{10}\big)}{4\sqrt{10}} e^{12\sqrt{10}} - \frac{3 \big(17 + 4\sqrt{10}\big)}{4\sqrt{10}},  14 e^{6\sqrt{10}} + \frac{37 - 8\sqrt{10}}{2\sqrt{10}} e^{12\sqrt{10}} - \frac{37 + 8\sqrt{10}\big)}{2\sqrt{10}} \Big],
\end{equation*}
\begin{equation*}
		J_2(u_1^*, u_2^*) = \frac{1}{K^2} \Big[\frac{14}{5} e^{6\sqrt{10}} + \frac{37 - 8\sqrt{10}}{10\sqrt{10}} e^{12\sqrt{10}} - \frac{37 + 8\sqrt{10}\big)}{10\sqrt{10}}, 
		\frac{33}{2} e^{6\sqrt{10}} + \frac{37 - 8\sqrt{10}}{8\sqrt{10}} e^{12\sqrt{10}} - \frac{37 + 8\sqrt{10}\big)}{8\sqrt{10}} \Big].
\end{equation*}
\begin{figure}[htbp]
	\centering
	\begin{minipage}[t]{0.49 \linewidth}
		\centering
		\includegraphics[width=0.85\linewidth]{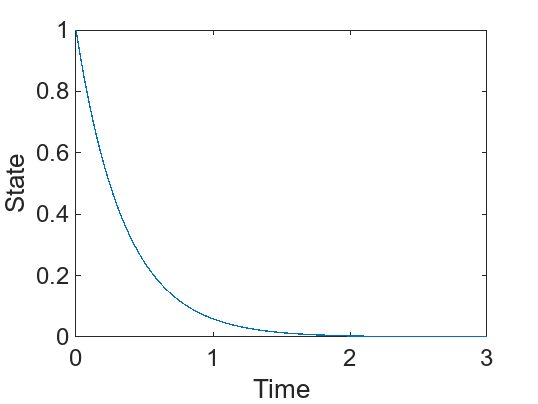}
		\caption*{\footnotesize{(a)\ Corresponding status}}
	\end{minipage}
	\begin{minipage}[t]{0.49 \linewidth}
		\centering
		\includegraphics[width=0.85\linewidth]{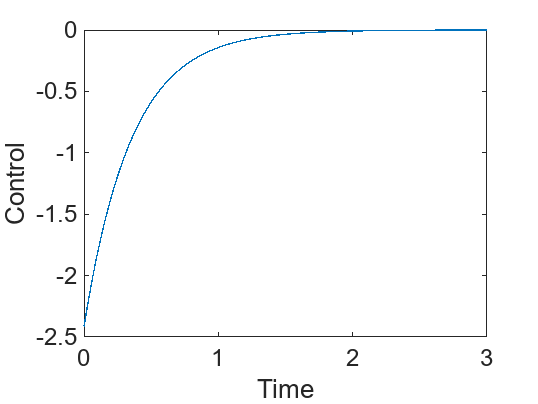}
		\caption*{\footnotesize{(b)\ Open-loop Pareto-Nash equilibrium solution}}
	\end{minipage}
	\centering
	\begin{minipage}[t]{0.49 \linewidth}
		\centering
		\includegraphics[width=0.85\linewidth]{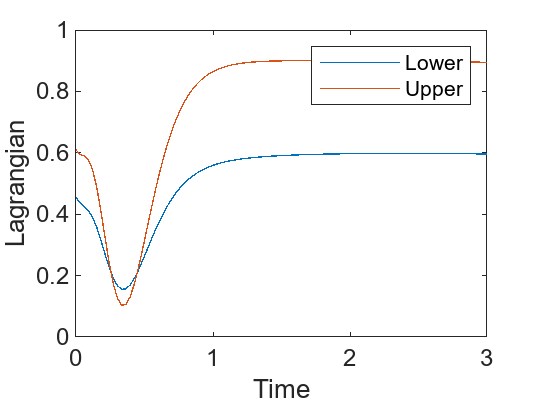}
		\caption*{\footnotesize{(c)\ Interval-valued Lagrange function of Player 1}}
	\end{minipage}
	\begin{minipage}[t]{0.49 \linewidth}
		\centering
		\includegraphics[width=0.85\linewidth]{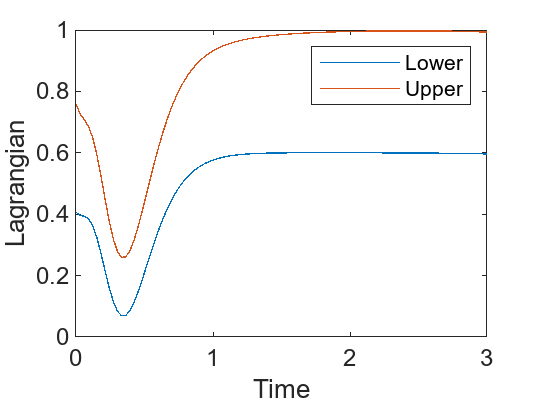}
		\caption*{\footnotesize{(d)\ Interval-valued Lagrange function of Player 2}}
	\end{minipage}
	\caption{Linear quadratic differential game (\ref{5.14})} 
	\label{fig:5.14}
\end{figure}

The open-loop Pareto-Nash equilibrium of the linear quadratic differential game (\ref{5.14}) is solved as
\[
u^*_1(t) = u^*_2(t) = \frac{2}{K} \Big(e^{\sqrt{8}t} - e^{6\sqrt{8} - \sqrt{8}t}\Big),
\]
and the corresponding status is
\[
x^*(t) = \frac{1}{K} \Big[\big(\sqrt{8} + 2\big) e^{\sqrt{8}t} + \big(\sqrt{8} - 2\big) e^{6\sqrt{8} - \sqrt{8}t}\Big]
\]
where $K = \big(\sqrt{8} + 2\big) + \big(\sqrt{8} - 2\big) e^{6\sqrt{8}}$.

The payment intervals corresponding to Players 1 and 2 are as follows.
\begin{equation*}
	\begin{split}
		J_1(u_1^*, u_2^*) = \frac{3}{5} \frac{1}{K^2} \Big[15 e^{6\sqrt{8}} + \frac{10 - 3\sqrt{8}}{2\sqrt{8}} e^{12\sqrt{8}} - \frac{10 + 3\sqrt{8}}{2\sqrt{8}}, 8 e^{6\sqrt{8}} + \frac{7 - 2\sqrt{8}}{\sqrt{8}} e^{12\sqrt{8}} - \frac{7 + 2\sqrt{8}}{\sqrt{8}} \Big],
	\end{split}
\end{equation*}
\begin{equation*}
	\begin{split}
		J_2(u_1^*, u_2^*) = \frac{1}{K^2} \Big[\frac{32}{5} e^{6\sqrt{8}} + \frac{14 - 4\sqrt{8}}{5\sqrt{8}} e^{12\sqrt{8}} - \frac{14 + 4\sqrt{8}}{5\sqrt{8}}, 7 e^{6\sqrt{8}} + \frac{10 - 3\sqrt{8}}{2\sqrt{8}} e^{12\sqrt{8}} - \frac{10 + 3\sqrt{8}}{2\sqrt{8}} \Big].
	\end{split}
\end{equation*}
\begin{figure}
	\centering
	\begin{minipage}[t]{0.49 \linewidth}
		\centering
		\includegraphics[width=0.85\linewidth]{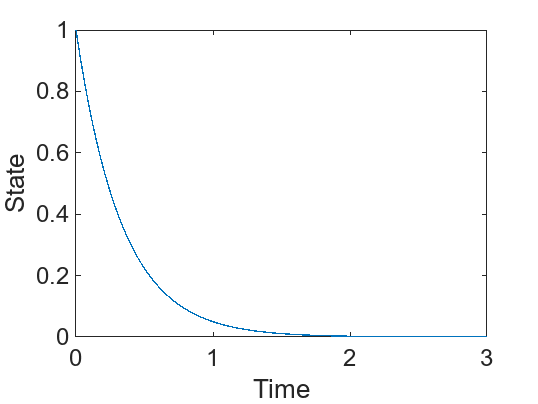}
		\caption*{\footnotesize{(a)\ Corresponding status}}
	\end{minipage}
	\begin{minipage}[t]{0.49 \linewidth}
		\centering
		\includegraphics[width=0.85\linewidth]{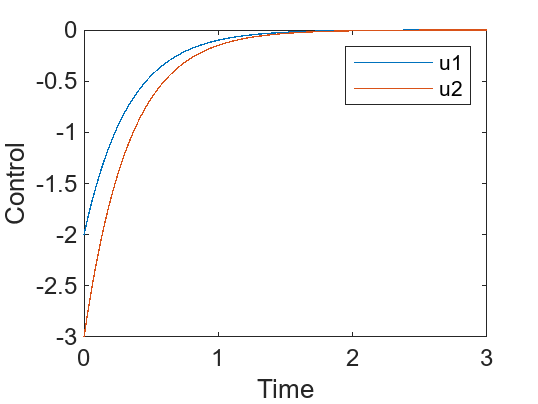}
		\caption*{\footnotesize{(b)\ Open-loop Pareto-Nash equilibrium solution}}
	\end{minipage}
	\centering
	\begin{minipage}[t]{0.49 \linewidth}
		\centering
		\includegraphics[width=0.85\linewidth]{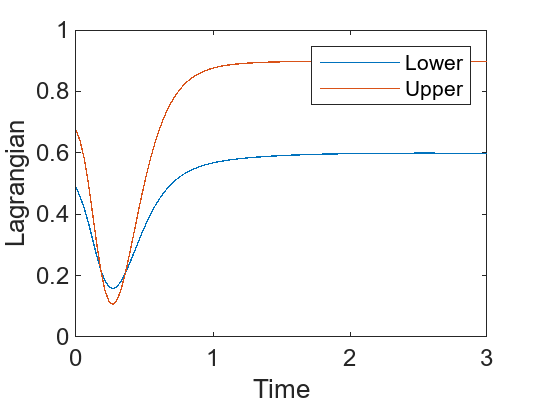}
		\caption*{\footnotesize{(c)\ Interval-valued Lagrange function of Player 1}}
	\end{minipage}
	\begin{minipage}[t]{0.49 \linewidth}
		\centering
		\includegraphics[width=0.85\linewidth]{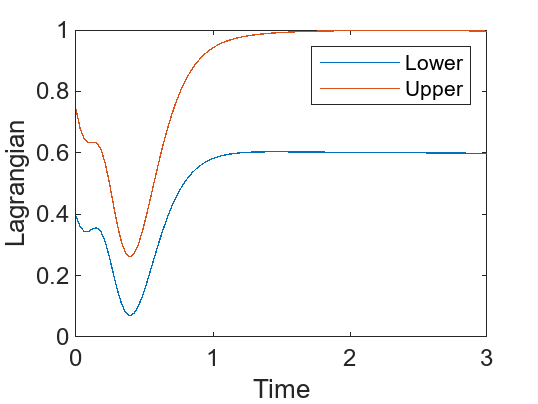}
		\caption*{\footnotesize{(d)\ Interval-valued Lagrange function of Player 2}}
	\end{minipage}
	\caption{Linear quadratic differential game (\ref{5.15})}
	\label{fig:5.15}
\end{figure}

The open-loop Pareto-Nash equilibrium of the linear quadratic differential game (\ref{5.15}) is $(u^*_1(t), u^*_2(t))$, where
\[
u^*_1(t) = \frac{2}{5 + e^{18}} (e^{3t} - e^{18 - 3t})
\ \mbox{and}\  
u^*_2(t) = \frac{3}{5 + e^{18}} (e^{3t} - e^{18 - 3t}).
\]
And the corresponding status is
\[
x^*(t) = \frac{1}{5 + e^{18}} (5 e^{3t} + e^{18 - 3t}).
\]

The payment intervals corresponding to Players 1 and 2 are as follows.
\begin{equation*}
	J_1(u_1^*, u_2^*) = \frac{3}{5} \frac{1}{(5 + e^{18})^2} \Big[\frac{39}{2} e^{18} + \frac{7}{24} e^{36} - \frac{79}{24},\  22 e^{18} + \frac{1}{2} e^{36} - \frac{9}{2} \Big],
\end{equation*}
\begin{equation*}
	J_2(u_1^*, u_2^*) = \frac{1}{(5 + e^{18})^2} \Big[\frac{14}{5} e^{18} + \frac{11}{30} e^{36} - \frac{59}{30},\  12 e^{18} + \frac{1}{2} e^{36} - \frac{7}{2} \Big].
\end{equation*}

The state curves of the linear quadratic differential game (\ref{5.13})-(\ref{5.15}), the control curves representing the open-loop Pareto-Nash equilibria, and the curves illustrating the interval Lagrange functions for both players are depicted in Figures \ref{fig:5.13}-\ref{fig:5.15}.
\end{example}

\vskip 25pt
\section{\bf Conclusions}
\vskip 10pt

The primary focus of this study is to analyze multi-objective interval differential games and their open-loop Pareto-Nash equilibrium. We consider the strategies of $n$ players as control processes, with each process located within a permissible control set that represents the strategy space. Consequently, the $n$-player multi-objective interval differential game can be viewed as a specialized form of an $n$-player multi-objective game with interval payoffs. Drawing from the literature \cite{LiLi}, we derive sufficient conditions for determining the existence of (weighted) open-loop Pareto-Nash equilibrium in multi-objective interval differential games. Additionally, by incorporating goal weights, we transform the $n$-player multi-goal interval differential game into an $n$-player single-goal interval differential game. For this modified version, we construct an interval-valued Hamiltonian function for each player to establish necessary conditions and sufficient conditions for identifying the open-loop Pareto-Nash equilibrium in the interval differential game. Finally, we convert the linear quadratic interval differential game into four classic linear quadratic differential games while providing a specific calculation method for determining its open-loop Pareto-Nash equilibrium.

This paper establishes the necessary and sufficient prerequisites exclusively for the existence of open-loop Pareto Nash equilibrium in interval differential games. Therefore, it is imperative to investigate the conditions for the existence of closed-loop Pareto-Nash equilibrium and feedback Pareto Nash equilibrium as well. Additionally, determining methodologies for computing closed-loop Pareto Nash equilibrium and feedback Pareto Nash equilibrium in linear quadratic interval differential games holds utmost importance. Furthermore, when addressing multi-objective differential games influenced by uncertain or imprecise information affecting both objective function and system state equation, integrating dynamic systems containing interval (fuzzy) dynamics into such interval (fuzzy) differential game models becomes indispensable. Our forthcoming research will primarily focus on establishing these types of differential game models along with their corresponding equilibrium conditions.

\end{document}